\documentclass{amsart}
\usepackage{amssymb,amsmath,amsfonts,amsthm,amsxtra,afterpage}
\usepackage[all, 2cell]{xy}
\UseAllTwocells
\usepackage{xcolor}
\usepackage{mdframed}
\usepackage[normalem]{ulem}
\usepackage{enumerate}
\usepackage{enumerate}

\newtheorem{theorem}{Theorem}
\newtheorem{lemma}[theorem]{Lemma}
\newtheorem{proposition}[theorem]{Proposition}
\newtheorem{corollary}[theorem]{Corollary}

\newtheorem{remark}{Remark}
\theoremstyle{definition}
\newtheorem{definition}{Definition}

\numberwithin{equation}{section}

\renewenvironment{proof}[1][\proofname]{\begin{trivlist}\item[\hskip \labelsep \itshape \bfseries #1{}\hspace{2ex}]}
{\qed\end{trivlist}}

\begin{document}
\title[Unirationality]{Unirationality is the same thing as Rational Connectedness in characteristic zero.}
\author{Stephen Maguire}
\email{stephen.j.maguire@outlook.com}

\begin{abstract}
    In this paper we prove that unirationality, rational connectedness and rational chain connectedness coincide for smooth projective varieties over a field $ k $ of characteristic zero.  Our approach uses the MRC fibration to show that if $ X $ is a smooth projective variety, then there exists a variety $ \operatorname{MU}(X) $, together with rational maps $ \pi: X \dashrightarrow \operatorname{MU}(X) $ and $ \lambda: \operatorname{MU}(X) \dashrightarrow \operatorname{MRC}(X) $, such that
    \begin{itemize}
        \item[i)] if $ \nu: X \dashrightarrow \operatorname{MRC}(X) $, then $ \lambda \circ \pi = \nu $ on the appropriate domains;
        \item[ii)] the very general fibres of $ \pi $ are unirational;
        \item[iii)] the very general fibres of $ \lambda $ are rationally connected but not unirational.
    \end{itemize}
    We then apply an induction argument to show that $ \operatorname{MU}(X) $ is birationally equivalent to $ \operatorname{MRC}(X) $.
\end{abstract}
\maketitle

\section{Introduction}
    An $ n $ dimensional variety $ Z $ over a field $ k $ is rational if it is birational to $ \mathbb{P}^{n}_{k} $, and unirational if there exists a generically finite, dominant, rational map $ \mathbb{P}^{n}_{k} \dashrightarrow Z $. If there exists a generically finite, dominant, separable, rational map $ \mathbb{P}^{n}_{k} \dashrightarrow Z $, then $ Z $ is separably unirational. Rational and separably unirational varieties have particularly nice properties, and they lie at the heart of many early questions in mathematics, such as path integrals along a Riemann surface, Diophantine equations, and others. Moreover, rational and separably unirational varieties admit very simple parameterizations.

    For varieties over a field of positive characteristic, it is necessary to impose various separability conditions in order for many of the desired properties to
    continue to hold. For example, Shioda \cite{Shioda} and Shioda and Katsura \cite{ShiodaKatsura} showed that there exist unirational varieties which are not separably unirational when the characteristic of the base field is positive (for example, hypersurfaces of the form $ \mathcal{Z}(\langle \sum_{i=0}^{m} x_{i}^{n} \rangle) $ where $ n>m+1$, the characteristic of $ k $ is $ p >0 $, $ m $ is odd and $ p^{\ell} \equiv -1 \mod{n} $ for some $ \ell \in \mathbb{N} $). In particular, the plurigenera of these varieties do not vanish.

    By L\"{u}roth's theorem, every separably unirational curve is rational, and by Castelnuovo's criterion, every smooth separably unirational surface is rational.
    One might ask whether every unirational variety is rational. Clemens and Griffiths proved in \cite{GriffithsClemens} that a smooth cubic threefold is unirational but not rational. The counterexample they constructed is a variety over a field of characteristic zero, obtained using the theory of abelian varieties together with Hodge theory to construct an invariant they called the intermediate Jacobian. Artin and Mumford constructed a threefold which is separably unirational but not rational over fields of characteristic $ p>2 $. Their example was obtained as the resolution of singularities of a double cover of a web of quadrics in $ \mathbb{P}^{3} $, and they used the Brauer group to show that it is not rational.

    A variety $ Z $ is rationally connected if there exists a dominant rational map $ \psi: \mathbb{P}^{1}_{k} \times M \dashrightarrow Z $ such that $ (\psi,\psi): \mathbb{P}^{1}_{k} \times \mathbb{P}^{1}_{k} \times M \dashrightarrow Z \times Z $ is dominant. While it is difficult to prove that a variety is rational or separably unirational, it is considerably simpler to prove that a variety is rationally connected. Koll\'{a}r, Miyaoka and Mori \cite{KollarMiyoakaMori} showed that every smooth Fano variety in characteristic zero is rationally connected. In \cite{KollarMiyoakaMori2} they showed that every rationally connected variety is rationally chain connected; that is, any two points $ x_{1}, x_{2} $ can be connected by a chain of irreducible rational curves $ C_{1} \cup \cdots \cup C_{m} $ such that $ x_{1} \in C_{1} $, $ C_{i} \cap C_{i+1} \ne \emptyset $ and $ x_{2} \in C_{m} $. In characteristic zero, a smooth variety is rationally chain connected if and only if it is rationally connected. If $ Z $ is a variety and there exists a morphism $ \phi: \mathbb{P}^{1}_{k} \to Z $ such that $ \phi^{\ast}(T_{Z}) \otimes \mathcal{O}_{\mathbb{P}^{1}_{k}}(-1) $ is generated
    by global sections, then $ Z $ is rationally connected. The image of such a morphism is called a very free curve.

    A variety $ Z $ over a field $ k $ of positive characteristic is separably rationally connected if there exists a dominant, separable rational map $ \psi: \mathbb{P}^{1}_{k} \times M \dashrightarrow Z $ such that $ (\psi, \psi): \mathbb{P}^{1}_{k} \times \mathbb{P}^{1}_{k} \times M \dashrightarrow Z \times Z $ is dominant. Rational chain connectedness does not necessarily imply separable rational connectedness for varieties over a field $ k $ of positive characteristic (see \cite[V.5.19]{Kollar}). However, the existence of a very free curve does imply separable rational connectedness. Moreover, every general smooth Fano complete intersection is separably rationally connected. This was proven in characteristic zero by \cite{KollarMiyoakaMori} and in positive characteristic by \cite{ChenZhu} and \cite{Tian}.

    If $ X $ is a variety over a field $ k $ of characteristic zero and $ X^{0} $ is an open subvariety of $ X $, then a proper morphism $ \pi: X^{0} \to Z^{0} $ is a \emph{rationally chain connected fibration} if its fibres are rationally chain connected and $ \pi_{\ast}(\mathcal{O}_{X^{0}}) \cong \mathcal{O}_{Z^{0}} $.  Such a morphism is a \emph{maximal rationally chain connected fibration} if, for every rationally chain connected fibration $ \pi_{1}: X^{1} \to Z^{1} $ on an open subvariety $ X^{1} $ of $ X $, there exists a rational map $ \tau: Z^{1} \dashrightarrow Z^{0} $ such that $ \pi = \tau \circ \pi_{1} $.  Campana and Koll\'{a}r, and Miyaoka and Mori, independently proved that the MRCC fibration $ \pi: X \dashrightarrow Z $ of a normal proper variety $ X $ exists (see \cite[Chapter IV, Rationally Connected Varieties, Section 5, Maximal Rationally Connected Fibrations, Theorem 5.2]{Kollar} for one such reference).  Over a field of characteristic zero, if $ X $ is a smooth proper variety, then we may assume that both $ Z $ and the fibres of $ \pi $ are smooth.  Because smooth rationally chain connected varieties are rationally connected in characteristic zero, this quotient is called the maximal rationally connected fibration, or MRC fibration.  The map $ \pi $ is called the \emph{MRC fibration of $ X $}, and $ Z $ is the \emph{MRC quotient of $ X $}.  The MRC fibration is unique up to birational equivalence.

    In this paper we prove that for every smooth projective variety $ X $ over a field $ k $ of characteristic zero, there exists a variety $ \operatorname{MU}(X) $, together with rational maps $ \pi: X \dashrightarrow \operatorname{MU}(X) $ and $ \lambda: \operatorname{MU}(X) \dashrightarrow \operatorname{MRC}(X) $, such that:
    \begin{itemize}
        \item[i)] if $ \nu: X \dashrightarrow \operatorname{MRC}(X) $ is the maximal rationally connected fibration, then $ \lambda \circ \pi = \nu $ on an appropriate domain;
        \item[ii)] the very general fibres of $ \pi $ are unirational;
        \item[iii)] the very general fibres of $ \lambda $ are rationally connected but not unirational.
    \end{itemize}
    We then use an induction argument to show that $ \operatorname{MRC}(X) $ is birationally equivalent to $ \operatorname{MU}(X) $.  Using this, we conclude that unirationality, rational connectedness and rational chain connectedness coincide for smooth varieties over a field of characteristic zero.
\section{Conventions} \label{S:conventions}
In this section we set out some conventions that we will use, including those for the machinery of schemes and varieties, following the notational conventions of Grothendieck.  The most important is that if $ f: X \to Y $ is a map of schemes and $ U \subset Y $ is an open subscheme, then we denote the map $ \mathcal{O}_{Y}(U) \to \mathcal{O}_{X}(f^{-1}(U)) $ by $ f^{\sharp} $.  If $ V $ is a subvariety of a variety $ X $, then we denote the closure of $ V $ in $ X $ by $ \overline{V} $.  We denote a rational map $ \phi $ from $ X $ to $ Y $ by $ \phi: X \dashrightarrow Y $.  Finally, if $ X $ is a variety, then we denote the function field of $ X $ by $ K(X) $.
\section{Preliminary Definitions and Recollections}
Most of the material in this section can be found in various parts of \cite{Kollar}.
\begin{definition}
    A projective variety $ Z $ over a field $ k $ of arbitrary characteristic is \emph{separably rationally connected} if there exist a variety $ M $ and a dominant, separable rational map $ \psi: \mathbb{P}^{1}_{k} \times M \dashrightarrow Z $ such that $ (\psi,\psi): \mathbb{P}^{1}_{k} \times \mathbb{P}^{1}_{k} \times M \dashrightarrow Z \times Z $ is dominant.  This is equivalent to the existence of a morphism $ \phi: \mathbb{P}^{1}_{k} \to Z $ such that $ \phi^{\ast}(T_{Z}) \otimes \mathcal{O}_{\mathbb{P}^{1}_{k}}(-1) $ is generated by global sections.
\end{definition}
\begin{definition}
    A projective variety $ Z $ over a field $ k $ is \emph{rationally chain connected} if any two points can be connected by a chain of rational curves.
\end{definition}
\begin{remark}
    Over a field of characteristic zero, a smooth rationally chain connected variety is rationally connected.
\end{remark}
\begin{definition}
    An $ n $-dimensional variety $ Z $ over a field $ k $ of arbitrary characteristic is \emph{separably unirational} if there exists a generically finite, dominant, separable rational map $ \psi: \mathbb{P}^{n}_{k} \dashrightarrow Z $.
\end{definition}
\begin{remark}
    Both separable rational connectedness and separable unirationality are birational properties.
\end{remark}
The following definitions and theorems concerning the maximal rationally connected fibration can be found in \cite[Chapter IV, Rationally Connected Varieties, Section 5, Maximally Rationally Connected Fibrations]{Kollar}.
\begin{definition}
    Let $ X $ be a normal proper variety and $ X^{0} \subset X $ an open subset.  A proper morphism $ \nu: X^{0} \to Z^{0} $ is called a \emph{rationally chain connected fibration} if the fibres of $ \nu $ are rationally chain connected and $ \nu_{\ast}(\mathcal{O}_{X^{0}}) = \mathcal{O}_{Z^{0}} $.  A proper morphism $ \nu: X^{0} \to Z^{0} $ is called a \emph{maximal rationally chain connected fibration} (or \emph{MRCC fibration}) if, for every open set $ X^{1} $ of $ X $ and every rationally chain connected fibration $ \nu_{1}: X^{1} \to Z^{1} $, there exists a rational map $ \tau: Z^{1} \dashrightarrow Z^{0} $ such that $ \nu \mid_{X^{0} \cap X^{1}} = \tau \circ \nu_{1} \mid_{X^{0} \cap X^{1}} $.
\end{definition}
Campana and Koll\'{a}r, and Miyaoka and Mori, independently proved the following theorem.
\begin{theorem}[Campana and Koll\'{a}r; Miyaoka and Mori]
    Let $ X $ be a normal proper variety.  Then the maximal rationally chain connected fibration $ X \dashrightarrow Z $ exists.
\end{theorem}
If the base field has characteristic zero and $ X $ is smooth, then we may shrink $ Z $ so that it is smooth.  Generic smoothness then ensures that every fibre is smooth and rationally connected.  In this case, we call the maximal rationally chain connected fibration the \emph{maximal rationally connected fibration}, or \emph{MRC fibration}, of $ X $.  Campana and Koll\'{a}r, and Miyaoka and Mori, independently proved the following theorem.
\begin{theorem}[Campana and Koll\'{a}r; Miyaoka and Mori]
    Let $ X $ be a smooth proper variety over a field of characteristic zero.  Then the MRC fibration $ \nu: X \dashrightarrow Z $ exists and is unique up to birational equivalence.
\end{theorem}
\begin{theorem}
    Let $ X_{1}, X_{2} $ be smooth proper varieties over a field of characteristic zero, and let $ f_{X}: X_{1} \dashrightarrow X_{2} $ be a dominant map.  Let $ \nu_{i}: X_{i} \dashrightarrow Z_{i} $ be the MRC fibrations.  Then there exists a rational map $ f_{Z}: Z_{1} \dashrightarrow Z_{2} $ such that
    \begin{equation*}
        f_{Z} \circ \nu_{1} = \nu_{2} \circ f_{X}.
    \end{equation*}
\end{theorem}
\section{Proof of the Main Theorems}
\begin{proposition} \label{P:startMU}
    If $ X $ is an $ n $-dimensional projective variety over a field $ k $ of arbitrary characteristic, where $ n \in \mathbb{N} $, then there exists a normal projective variety $ M $ such that
    \begin{itemize}
        \item[a)] $ M $ is not separably uniruled;
        \item[b)] $ \dim(M) \le \dim(X) $, with equality if and only if $ X $ is not separably uniruled;
        \item[c)] if $ \dim(M) = \ell < n $, then there exists a separable, generically finite, dominant, rational map $ \phi: M \times \prod_{i=1}^{n-\ell} \mathbb{P}^{1}_{k} \dashrightarrow X $.
    \end{itemize}
\end{proposition}
\begin{proof}
    If $ X $ is not separably uniruled, let $ M = X $.  If $ X $ is separably uniruled, then there exist an $ (n-1) $-dimensional variety $ M_{1} $ and a separable, dominant, generically finite, rational map $ \phi_{1}: M_{1} \times \mathbb{P}^{1} \dashrightarrow X $.  We induct on the dimension $ n $ of $ X $ to prove the proposition for an $ n $-dimensional separably uniruled projective variety $ X $.  If $ n = 1 $, then $ X $ is a rational curve, so the normalization map is a separable, generically finite, dominant, rational map from $ \mathbb{P}^{1}_{k} $ to $ X $.  Therefore the proposition holds when $ n = 1 $.

    Assume that the proposition holds for every separably uniruled variety of dimension $ n < N $, and let $ X $ be a separably uniruled variety of dimension $ N $.  Since $ X $ is separably uniruled, there exist a normal variety $ M_{1} $ and a separable, generically finite, dominant, rational map $ \phi_{1}: M_{1} \times \mathbb{P}^{1}_{k} \dashrightarrow X $.  If $ M_{1} $ is not separably uniruled, then $ M_{1} $ and $ \phi_{1} $ satisfy the requirements of the proposition.  If $ M_{1} $ is separably uniruled, then there exist an $ \ell $-dimensional (with $ \ell < N-1 $) normal, non-separably-uniruled variety $ M $ and a separable, generically finite, dominant, rational map $ \phi_{2}: M \times \prod_{i=1}^{N-\ell-1} \mathbb{P}^{1}_{k} \dashrightarrow M_{1} $.  Setting $ \phi $ equal to the rational map $ \phi_{1} \circ (\phi_{2}, \operatorname{id}_{\mathbb{P}^{1}_{k}}) $ from $ M \times \prod_{i=1}^{N-\ell} \mathbb{P}^{1}_{k} $ to $ X $, we find that $ M $ and $ \phi $ satisfy the requirements of the proposition.
\end{proof}
\begin{definition}
    Let $ X $ be a smooth variety with function field $ K(X) $.  A \emph{unirational fibration} is a rational map $ \pi: X \dashrightarrow Y $ whose very general fibres are unirational and such that, if $ \nu: X \dashrightarrow \operatorname{MRC}(X) $ is the MRC fibration, then there exists a rational map $ \lambda: Y \dashrightarrow \operatorname{MRC}(X) $ with $ \lambda \circ \pi = \nu $ on an appropriate domain.  A unirational fibration $ \pi: X \dashrightarrow Y $ is a \emph{maximal unirational fibration} if the fibres of $ \lambda $ are rationally connected but not unirational.
\end{definition}
\begin{lemma} \label{L:unirationFib}
    Let $ X $ be a smooth $ n $-dimensional variety.  The following are equivalent:
    \begin{itemize}
        \item[i)] there exists an $ s $-dimensional variety $ M $ such that:
            \begin{itemize}
                \item[a)] there exist $ n-s $ elements $ z_{1},\dots,z_{n-s} $, transcendental over $ K(M) $, such that $ K(M)(z_{1},\dots,z_{n-s}) $ is a finite extension of $ K(X) $; that is, the following inclusions of fields hold:
                    \begin{equation*}
                    \xymatrix{
                        K(M)(z_{1},\dots,z_{n-s}) \ar@{-}[d] \\
                        K(X) \ar@{-}[d] \\
                        K(M)
                    }
                \end{equation*}
                \item[b)] the function field of any maximal rationally connected fibration of $ X $ is contained in $ K(M) $.
            \end{itemize}
        \item[ii)] there exists a rational map $ \pi: X \dashrightarrow M $ that is a unirational fibration.
    \end{itemize}
\end{lemma}
\begin{proof}
    The inclusions of $ K(M) $ in $ K(X) $ and of $ K(\operatorname{MRC}(X)) $ in $ K(M) $ are equivalent to the existence of rational maps $ \pi: X \dashrightarrow M $ and $ \lambda: M \dashrightarrow \operatorname{MRC}(X) $ such that, if $ \nu: X \dashrightarrow \operatorname{MRC}(X) $ is a maximal rationally connected fibration, then $ \lambda \circ \pi = \nu $ on an appropriate domain.  Let $ m $ be a very general closed point of $ M $.  Because $ X $ is smooth, we may assume by generic smoothness that $ \pi $ is smooth of relative dimension $ n-s $ \cite[Chapter III, Cohomology, Section 10, Smooth Morphisms, Corollary 10.7]{HartshorneAG}.  Since every irreducible component of $ \pi^{-1}(m) $ is $ (n-s) $-dimensional, the requirement that each irreducible component $ W_{i} $ of $ \pi^{-1}(m) $ be unirational is equivalent to the existence of a generically finite, dominant, rational map from $ \mathbb{P}^{n-s}_{k} $ to $ W_{i} $.  The existence of such dominant rational maps is in turn equivalent to the existence of $ n-s $ elements $ z_{1},\dots,z_{n-s} $, transcendental over $ K(M) $, such that $ K(M)(z_{1},\dots,z_{n-s}) $ is a finite extension of $ K(X) $.  Therefore i) and ii) are equivalent.
\end{proof}
\begin{proposition} \label{P:unirationalFibBegin}
    Let $ X $ be a smooth projective uniruled variety over a field $ k $ of characteristic zero.  If $ M $ is the variety from Proposition~\ref{P:startMU}, then there exist a generically finite, dominant, rational map $ \iota: M \to N $ and a birational map $ \pi: \widehat{X} \to X $ such that there is a morphism $ \widehat{\Psi}: \widehat{X} \to N $ whose very general fibres are unirational, generically smooth and connected.
\end{proposition}
\begin{proof}
    Because $ X $ is uniruled, Proposition~\ref{P:startMU} shows that there exist a normal non-uniruled variety $ M $ of dimension $ \ell $, with $ 0 \le \ell < n $, and a generically finite, dominant, rational map $ \phi: M \times \prod_{i=1}^{n-\ell} \mathbb{P}^{1}_{k} \dashrightarrow X $.  We may assume that $ M $ is non-singular, since resolutions of singularities exist in characteristic zero and any resolution of the singularities of $ M $ is birational to $ M $.

    Let $ U $ be the largest open set of $ M \times \prod_{i=1}^{n-\ell} \mathbb{P}^{1}_{k} $ on which $ \phi $ is defined, and let $ p_{1}: M \times \prod_{i=1}^{n-\ell} \mathbb{P}^{1}_{k} \to M $ be the natural projection onto the first factor.  Define a morphism $ \psi: U \to \operatorname{Hilb}_{n-\ell}(X) $ as follows.  For $ m \in U $, let $ \psi(m) $ be the point of $ \operatorname{Hilb}_{n-\ell}(X) $ corresponding to $ \overline{\phi(p_{1}^{-1}(m) \cap U)} $, which we denote by $ X_{m} $.  Because $ p_{1}^{-1}(m) \cap U $ is an open subvariety of $ \prod_{i=1}^{n-\ell} \mathbb{P}^{1}_{k} $, the map
    \begin{equation*}
        \phi_{m}: m \times \prod_{i=1}^{n-\ell} \mathbb{P}^{1}_{k} \dashrightarrow X_{m}
    \end{equation*}
    is a dominant, generically finite, rational map, and so $ X_{m} $ is unirational.

    Let $ \operatorname{Univ}_{n-\ell}(X) $ be the universal family over $ \operatorname{Hilb}_{n-\ell}(X) $, and, using the structure map $ \psi $, denote by $ Y $ the fibre product
    \begin{equation*}
        Y:=U \times_{\operatorname{Hilb}_{n-\ell}(X)} \operatorname{Univ}_{n-\ell}(X).
    \end{equation*}
    There is a morphism $ \rho: \operatorname{Univ}_{n-\ell}(X) \to X $ sending an $ (n-\ell) $-dimensional subscheme of $ X $ to its image in $ X $.  Let $ p_{2,Y}: Y \to \operatorname{Univ}_{n-\ell}(X) $ be the natural projection.  The image of $ Y $ under $ \rho \circ p_{2,Y} $ is a constructible $ n $-dimensional subset of $ X $, hence dense in $ X $.  Since a dense constructible set contains an open set, there is an open subvariety $ X_{0} $ contained in the image of $ \rho \circ p_{2,Y} $.

    Define an algebraic relation on the closed points of $ U $ as follows.  For $ m_{1},m_{2} \in U $, set $ m_{1} \sim m_{2} $ if $ X_{m_{1}} \cap X_{m_{2}} $ is nonempty.  Let $ R $ be the subvariety of $ U \times U $ whose points are the pairs $ (m_{1},m_{2}) $ with $ m_{1} \sim m_{2} $.  Let $ V $ be an irreducible subscheme of maximal dimension of the schematic locus of the algebraic space $ [U/R] $.  The algebraic space $ [U/R] $ is separated, and hence so is $ V $.  Equipping $ V $ with its reduced induced scheme structure, we find that $ V $ is a variety.
    
    If $ \iota $ is the natural map of fppf presheaves from $ U $ to $ [U/R] $, then we denote the morphism $ \iota \circ p_{1,Y} $ by $ \widetilde{\Psi} $ and the morphism $ \rho \circ p_{2,Y} $ by $ \Phi $.  The following diagram then commutes:
    \begin{equation} \label{E:1}
    \xymatrix{
        Y:= U \times_{\operatorname{Hilb}_{n-\ell}} \operatorname{Univ}_{n-\ell} \ar@/_1pc/[dd]_{\widetilde{\Psi}} \ar@{-->}[rr]^{\Phi} \ar[d]^{p_{1,Y}} & & X \supseteq X_{0} \\
        U \ar[d]^{\iota} \ar[r] & M \\
        V \subseteq [U/R]
        }
    \end{equation}
    We now replace $ V $ by its non-singular locus.  After making this replacement, we shrink $ U $ and $ V $ so that both are generically smooth and $ \iota $ is generically smooth as well, and we replace $ Y $ by $ \widetilde{\Psi}^{-1}(V) \cap \Phi^{-1}(X_{0}) $.  We claim that, for every $ x \in X_{0} $, every point of $ \Phi^{-1}(x) $ maps under $ \widetilde{\Psi} $ to a single point of $ V $.  The points of $ \Phi^{-1}(x) $ are the pairs $ (m,x) $ with $ x \in X_{m} $, so $ m_{1} \sim m_{2} $ for any two points $ (m_{1},x),(m_{2},x) $ of $ \Phi^{-1}(x) $.  Consequently, $ \widetilde{\Psi} $ collapses the fibre $ \Phi^{-1}(x) $ to a single point.  Let $ \Psi: X_{0} \to V $ be the morphism of varieties sending a point $ x $ to $ \widetilde{\Psi}(\Phi^{-1}(x)) $.

    By \cite[I, Varieties, Section 4, Rational Maps, Proposition 4.9]{HartshorneAG}, every variety is birational to a hypersurface in some projective space.  In particular, $ V $ is birational to a projective variety $ N $.  We may accordingly update the diagram in~\eqref{E:1} to the following:
    \begin{equation} \label{E:2}
    \xymatrix{
        & Y \ar[rr]^{\Phi} \ar@/^1pc/[dd]^{\widetilde{\Psi}} \ar[d]_{p_{1,Y}} & & X_{0} \subseteq X \ar@{-->}[ddll]^{\Psi} \\
        M \ar@{-->}[d]^{\iota} & U \ar[l] \ar[d]_{\iota} \\
        N & \ar@{-->}[l] V
         }.
    \end{equation}
    Resolving the indeterminacies of the rational map $ \Psi: X \dashrightarrow N $ by blow-ups, we obtain a morphism $ \widehat{\Psi}: \widehat{X} \to N $.  By our earlier work, the fibres of $ \widehat{\Psi} $ are unirational.  Since resolutions of singularities exist in characteristic zero, generic smoothness lets us assume that $ \widehat{\Psi} $ is generically smooth.  The very general fibres $ \widehat{\Psi}^{-1}(n) $ are of the form $ \bigcup_{m} X_{m} $.  Each $ X_{m} $ is connected, being the continuous image of $ \prod_{i=1}^{n-\ell} \mathbb{P}^{1}_{k} $.  Moreover, for any $ m_{1},m_{2} \in \iota^{-1}(n) $ we have $ X_{m_{1}} \cap X_{m_{2}} \ne \emptyset $.  Hence the very general fibres are connected.
\end{proof}
\begin{proposition} \label{P:mrcEquiv}
    Let $ X $ be a uniruled smooth projective variety over a field $ k $ of characteristic zero.  If $ \widehat{\Psi}: \widehat{X} \to N $ is the morphism from Proposition~\ref{P:unirationalFibBegin}, then
    \begin{align*}
        \operatorname{MRC}(N) & \cong \operatorname{MRC}(\widehat{X}) \\
        & \cong \operatorname{MRC}(X).
    \end{align*}
\end{proposition}
\begin{proof}
    Since $ \widehat{X} $ is birational to $ X $,
    \begin{equation*}
        \operatorname{MRC}(X) \cong \operatorname{MRC}(\widehat{X}).
    \end{equation*}

    Let $ U $ be an open subvariety of $ N $ such that $ \widehat{\Psi}: \widehat{\Psi}^{-1}(U) \to U $ is generically smooth and the fibres of $ \widehat{\Psi} $ over points $ u \in U $ are unirational.  Let \linebreak $ \gamma: N \to \operatorname{MRC}(N) $, and let $ V $ be an open subvariety contained in $ \gamma(U) $.  We claim that $ \gamma \circ \widehat{\Psi}: \widehat{X} \dashrightarrow \operatorname{MRC}(N) $ is a maximal rationally connected fibration.  Denote $ \gamma \circ \widehat{\Psi} $ by $ \nu $.

    The morphism $ \nu $ is proper because $ \widehat{X} $ and $ \operatorname{MRC}(N) $ are projective.  By the Stein factorization theorem, $ \widehat{\Psi}_{\ast}(\mathcal{O}_{\widehat{X}}) \cong \mathcal{O}_{N} $.  Since $ \gamma: N \to \operatorname{MRC}(N) $ is a rationally connected fibration, $ \nu_{\ast}(\mathcal{O}_{N}) \cong \mathcal{O}_{\operatorname{MRC}(N)} $, and so $ \nu_{\ast}(\mathcal{O}_{\widehat{X}}) \cong \mathcal{O}_{\operatorname{MRC}(N)} $.  Therefore $ \nu: \widehat{X} \to \operatorname{MRC}(N) $ is a rationally connected fibration.

    We claim that $ \operatorname{MRC}(N) $ is a maximal rationally connected quotient of $ \widehat{X} $ under $ \nu $.  Suppose that $ Z $ is a rationally connected subvariety such that $ Z \cap \nu^{-1}(m) $ is nonempty for some $ m \in V $, but $ Z \not\subseteq \nu^{-1}(m) $.

    In this case there is a rational curve $ C_{0} $ in $ Z $ such that $ C_{0} \cap \nu^{-1}(m) $ is nonempty, but $ C_{0} \not\subseteq \nu^{-1}(m) $.  The morphism $ \nu $ does not contract $ C_{0} $, since $ C_{0} \not\subseteq \nu^{-1}(m) $.  Denote the image of $ C_{0} $ by $ C_{2} $.  Because $ \nu $ does not contract $ C_{0} $, neither does $ \widehat{\Psi} $.  Consequently, there is a curve $ C_{1} \subseteq N $ that is the image of $ C_{0} $ under $ \widehat{\Psi} $ and such that $ C_{2} $ is the image of $ C_{1} $ under $ \gamma $.  By \cite[II, Schemes, Section 6, Divisors, Proposition 6.8]{HartshorneAG} and L\"{u}roth's theorem, the curve $ C_{1} $ is rational.  This means that $ C_{1} \cap \gamma^{-1}(m) \ne \emptyset $.  But then
    \begin{align*}
        C_{0} & \subseteq \Psi^{-1}(C_{1}) \\
        & \subseteq (\gamma \circ \Psi)^{-1}(m) \\
        &= \nu^{-1}(m),
    \end{align*}
    a contradiction.  Hence $ Z $ is contained in $ \nu^{-1}(m) $.

    Now assume that $ \nu_{1}: \widehat{X} \dashrightarrow W $ is another rationally connected fibration, with domain $ X_{1} $, and let $ (\nu,\nu_{1}): X_{0} \cap X_{1} \to \operatorname{MRC}(N) \times W $.  Let $ \Gamma $ be the closure of $ (\nu,\nu_{1})(X_{0} \cap X_{1}) $.  If $ x_{1},x_{2} $ are two points of $ X_{0} \cap X_{1} $ with $ \nu_{1}(x_{1}) = \nu_{1}(x_{2}) $, then $ x_{1} $ and $ x_{2} $ can be connected by a chain of rational curves; that is, there are rational curves $ \mathcal{C}_{1},\dots,\mathcal{C}_{s} $ such that $ x_{1} \in \mathcal{C}_{1} $, $ x_{2} \in \mathcal{C}_{s} $ and $ \mathcal{C}_{i} \cap \mathcal{C}_{i+1} \ne \emptyset $ for $ 1 \le i \le s-1 $.

    By our earlier work, $ \mathcal{C}_{1} $ is contained in $ \nu^{-1}(\nu(x_{1})) $, since $ x_{1} \in \mathcal{C}_{1} \cap \nu^{-1}(\nu(x_{1})) $.  Note that
    \begin{align*}
        \emptyset & \ne \mathcal{C}_{1} \cap \mathcal{C}_{2} \\
        & \subseteq \nu^{-1}(\nu(x_{1})) \cap \mathcal{C}_{2},
    \end{align*}
    so $ \mathcal{C}_{2} \subseteq \nu^{-1}(\nu(x_{1})) $.  Assume that $ \mathcal{C}_{i} \subseteq \nu^{-1}(\nu(x_{1})) $ for $ 1 \le i < j \le s $.  Note that
    \begin{align*}
        \emptyset & \ne \mathcal{C}_{j-1} \cap \mathcal{C}_{j} \\
        & \subseteq \nu^{-1}(\nu(x_{1})) \cap \mathcal{C}_{j},
    \end{align*}
    so $ \mathcal{C}_{j} \subseteq \nu^{-1}(\nu(x_{1})) $.  By induction, $ \mathcal{C}_{i} \subseteq \nu^{-1}(\nu(x_{1})) $ for $ 1 \le i \le s $.  The points $ \nu(x_{1}) $ and $ \nu(x_{2}) $ coincide, since $ x_{2} \in \mathcal{C}_{s} \cap \nu^{-1}(\nu(x_{1})) $.  Hence the morphism $ p_{2,\Gamma}: \Gamma \to W $ is generically one-to-one, and so $ \Gamma $ and $ W $ are birational.  If $ \alpha: W \dashrightarrow \Gamma $ is a rational map and $ p_{1,\Gamma}: \Gamma \to \operatorname{MRC}(N) $ is the projection onto the first factor, then $ p_{1,\Gamma} \circ \alpha $ is a rational map from $ W $ to $ \operatorname{MRC}(N) $.  This shows that $ \nu: \widehat{X} \to \operatorname{MRC}(N) $ is a maximal rationally connected fibration.  Hence
    \begin{equation*}
        \operatorname{MRC}(\widehat{X}) \cong \operatorname{MRC}(N).
    \end{equation*}
\end{proof}
\begin{theorem} \label{T:maxUnirationalFib}
    If $ X $ is a smooth projective variety over a field $ k $ of characteristic zero, then a maximal unirational fibration exists.  Moreover, if $ M $ is the variety from Proposition~\ref{P:startMU} and $ \operatorname{MU}(X) $ is a maximal unirational fibration, then $ \dim(M) \ge \dim(\operatorname{MU}(X)) $.
\end{theorem}
\begin{proof}
    Let $ \nu: X \dashrightarrow \operatorname{MRC}(X) $ be the MRC fibration of $ X $.  Let $ \Psi: X \dashrightarrow N $ be the fibration from Proposition~\ref{P:unirationalFibBegin}, and let $ \gamma: N \to \operatorname{MRC}(N) \cong \operatorname{MRC}(X) $.  The fibres of $ \Psi $ are unirational, and the fibres of $ \gamma $ are rationally connected but possibly not unirational.  Let $ S $ be the set of rational maps $ \pi: X \dashrightarrow Y $ whose very general fibres are unirational.  For each such $ \pi: X \dashrightarrow Y $, there automatically exists a map $ \alpha: Y \dashrightarrow \operatorname{MRC}(X) $.  The set $ S $ is nonempty, since $ \Psi: X \dashrightarrow N $ belongs to it.  We order $ S $ as follows: if $ \pi_{1}: X \dashrightarrow Y_{1} $ and $ \pi_{2}: X \dashrightarrow Y_{2} $ are two elements of $ S $, then $ \pi_{1} \preceq \pi_{2} $ if $ K(Y_{1}) $ is an extension of $ K(Y_{2}) $.

    This relation is clearly reflexive, antisymmetric and transitive.  Let
    \begin{equation*}
        \pi_{1} \preceq \pi_{2} \preceq \cdots
    \end{equation*}
    be a chain of elements of $ S $, where $ \pi_{i}: X \dashrightarrow Y_{i} $, so that
    \begin{equation*}
        K(Y_{1}) \supseteq K(Y_{2}) \supseteq \cdots.
    \end{equation*}
    Each $ K(Y_{i}) $ contains $ K(\operatorname{MRC}(X)) $ for all $ i \in \mathbb{N} $, so $ \cap_{i=1}^{\infty} K(Y_{i}) $ is a field $ L $ containing $ K(\operatorname{MRC}(X)) $.  The transcendence degree of $ L $ is bounded below by that of $ K(\operatorname{MRC}(X)) $.  Consequently, there is an $ m \in \mathbb{N} $ such that the transcendence degree of $ K(Y_{m}) $ is equal to that of $ K(Y_{i}) $ for $ i \ge m $.  Denote $ \operatorname{trdeg}_{k}(K(Y_{m})) $ by $ c $.

    We claim, moreover, that there is a $ j \ge m $ such that $ K(Y_{j}) = L $.  Since $ L $ is a finite extension of $ K(Y_{m}) $ and $ K(Y_{i}) $ is a finite extension of $ K(Y_{i+1}) $ for $ i \ge m $, such a $ j $ must exist.

    By Lemma~\ref{L:unirationFib}, there are $ n-c $ elements $ z_{1},\dots,z_{n-c} $ such that $ K(Y_{j})(z_{1},\dots,z_{n-c}) $ is a finite extension of $ K(X) $.  By the same lemma, there is a variety $ W $ such that
    \begin{align*}
        K(W) &\cong K(Y_{j}) \\
        & \cong L,
    \end{align*}
    together with a rational map $ \tau: X \dashrightarrow W $ whose very general fibres are unirational.  Since $ K(\operatorname{MRC}(X)) \subseteq K(Y_{i}) $ for all $ i \in \mathbb{N} $,
    \begin{align*}
        K(\operatorname{MRC}(X)) & \subseteq \cap_{i=1}^{\infty} K(Y_{i}) \\
        &=L.
    \end{align*}
    Therefore there is a rational map $ \lambda: W \dashrightarrow \operatorname{MRC}(X) $ such that $ \lambda \circ \tau = \nu $ on an appropriate domain.  Hence $ \tau: X \dashrightarrow W $ belongs to $ S $, and $ \pi_{i} \preceq \tau $ for all $ i \in \mathbb{N} $.  This shows that every chain in $ S $ has an upper bound, so a maximal element exists by Zorn's lemma.

    Let $ \pi_{1}: X \dashrightarrow W_{1} $ and $ \pi_{2}: X \dashrightarrow W_{2} $ be two maximal elements.  Let $ L $ be the field $ K(W_{1}) \cap K(W_{2}) $ and let $ s $ be the minimum of $ \operatorname{trdeg}_{k}(K(W_{1})) $ and $ \operatorname{trdeg}_{k}(K(W_{2})) $; then $ \operatorname{trdeg}_{k}(L) = s $.  Without loss of generality, assume that $ s $ is the transcendence degree of $ K(W_{1}) $. 

    Let $ W $ be a variety such that $ K(W) \cong L $.  There exist rational maps
    \begin{align*}
        \alpha_{1}&: W_{1} \dashrightarrow W \\
        \alpha_{2}&: W_{2} \dashrightarrow W,
    \end{align*}
    such that the following diagram commutes:
    \begin{equation*}
    \xymatrix{
        & X \ar@{-->}[dl]_{\pi_{1}} \ar@{-->}[dr]^{\pi_{2}} \\
        W_{1} \ar@{-->}[dr]_{\alpha_{1}} & & W_{2} \ar@{-->}[dl]^{\alpha_{2}} \\
        & W
        }.
    \end{equation*}
    Write $ \pi := \alpha_{1} \circ \pi_{1} = \alpha_{2} \circ \pi_{2} $ for the resulting rational map $ X \dashrightarrow W $.  The fibres of $ \pi_{1} $ and $ \pi_{2} $ are unirational and $ \alpha_{1} $ is finite, so the fibres of $ \pi $ are unirational as well.  Therefore $ \pi: X \dashrightarrow W $ is an element of $ S $, and $ \pi_{1} \preceq \pi $ and $ \pi_{2} \preceq \pi $.  Because $ \pi_{1} $ and $ \pi_{2} $ are both maximal elements of $ S $,
    \begin{align*}
        \pi_{1} &= \pi \\
        &= \pi_{2}.
    \end{align*}
    Hence there is a unique maximal element.  Let $ \pi: X \dashrightarrow \operatorname{MU}(X) $ be this element; it is unique up to birational equivalence.

    We claim that if $ \lambda: \operatorname{MU}(X) \dashrightarrow \operatorname{MRC}(X) $ is the rational map such that $ \lambda \circ \pi = \nu $ on an appropriate domain, then the very general fibres of $ \lambda $ are rationally connected but not unirational.  Because $ \lambda \circ \pi = \nu $ on an appropriate domain, the very general fibres of $ \lambda $ are rationally connected.  If they were unirational, then $ \pi $ would not be a maximal element of $ S $.  Therefore the very general fibres of $ \lambda $ are not unirational, and $ \pi: X \dashrightarrow \operatorname{MU}(X) $ is a maximal unirational fibration.  Because $ \Psi: X \dashrightarrow N $ belongs to $ S $, the dimension of $ \operatorname{MU}(X) $ is at most that of $ N $.  Since $ \dim(N) \le \dim(M) $, the dimension of $ \operatorname{MU}(X) $ is at most that of $ M $.
\end{proof}
\begin{theorem} \label{T:main}
    If $ X $ is a smooth projective variety over a field $ k $ of characteristic zero, then $ X $ is unirational if and only if it is rationally chain connected; that is, rational chain connectedness, rational connectedness and unirationality are equivalent for a smooth projective variety over a field $ k $ of characteristic zero.
\end{theorem}
\begin{proof}
    We prove the theorem by induction on the dimension of $ X $.  For a smooth projective variety over a field $ k $ of characteristic zero, rational connectedness and rational chain connectedness are equivalent, and for curves and surfaces, rational connectedness, rational chain connectedness and unirationality coincide.

    Assume that every smooth projective rationally connected variety of dimension less than $ n $ is unirational.  Let $ X $ be a smooth projective rationally connected variety of dimension $ n $.  Because $ X $ is rationally connected, $ \operatorname{MRC}(X) $ equals $ \operatorname{Spec}(k) $, and $ X $ is uniruled.  By Theorem~\ref{T:maxUnirationalFib}, there is a variety $ \operatorname{MU}(X) $ together with rational maps $ \pi: X \dashrightarrow \operatorname{MU}(X) $ and $ \lambda: \operatorname{MU}(X) \to \operatorname{MRC}(X) $ such that the very general fibres of $ \pi $ are unirational, the very general fibres of $ \lambda $ are rationally connected but not unirational, and $ \lambda \circ \pi $ is the maximal rationally connected fibration.  Since $ \operatorname{MRC}(X) $ equals $ \operatorname{Spec}(k) $, the very general fibre of $ \lambda $ is all of $ \operatorname{MU}(X) $, so $ \operatorname{MU}(X) $ is rationally connected.  Since $ X $ is uniruled, Theorem~\ref{T:maxUnirationalFib} gives
    \begin{align*}
        \dim(\operatorname{MU}(X)) &< \dim(X) \\
        &= n.
    \end{align*}
    If $ \dim(\operatorname{MU}(X)) > 0 $, then the induction hypothesis shows that $ \operatorname{MU}(X) $ is unirational.  But $ \operatorname{MU}(X) $, being the very general fibre of $ \lambda $, is not unirational, a contradiction.  Hence $ \operatorname{MU}(X) $ also equals $ \operatorname{Spec}(k) $.  Because the very general fibres of $ \pi: X \dashrightarrow \operatorname{MU}(X) $ are unirational and $ \operatorname{MU}(X) $ equals $ \operatorname{Spec}(k) $, the variety $ X $ is unirational.
\end{proof}
\begin{corollary}
    If $ X $ is a rationally connected variety over a field of characteristic zero, then $ X $ is unirational.
\end{corollary}
\begin{proof}
    By \cite[I, Varieties, Section 4, Rational Maps, Proposition 4.9]{HartshorneAG}, every variety is birational to a projective variety, so we may assume that $ X $ is projective.  Although rational chain connectedness is not a birational property, rational connectedness is; hence, if $ \widetilde{X} $ is a resolution of the singularities of $ X $, then $ \widetilde{X} $ is rationally connected.  By Theorem~\ref{T:main}, $ \widetilde{X} $ is unirational.  Since unirationality is a birational property, $ X $ is unirational.
\end{proof}
\bibliographystyle{amsplain}
\bibliography{UnirationalFib}
\end{document}